\documentclass[11pt,twoside, leqno]{article}

\usepackage{mathrsfs}
\usepackage{amssymb}
\usepackage{amsmath}
\usepackage{amsthm}
\usepackage{amsfonts}
\usepackage[colorlinks=true, pdfstartview=FitV, linkcolor=blue, citecolor=blue, urlcolor=blue]{hyperref}
\usepackage{color}

\allowdisplaybreaks

\def\rr{{\mathbb R}}
\def\rn{{{\rr}^n}}
\def\zz{{\mathbb Z}}

\def\nn{{\mathbb N}}
\def\fz{\infty}

\def\cm{{\mathcal M}}

\def\ls{\lesssim}

\def\r{\right}
\def\lf{\left}

\def\bint{{\ifinner\rlap{\bf\kern.30em--}
\int\else\rlap{\bf\kern.35em--}\int\fi}\ignorespaces}

\def\sbint{{\ifinner\rlap{\bf\kern.32em--}
\hspace{0.078cm}\int\else\rlap{\bf\kern.45em--}\int\fi}\ignorespaces}

\newtheorem{theorem}{Theorem}[section]
\newtheorem{lemma}[theorem]{Lemma}
\newtheorem{corollary}[theorem]{Corollary}

\theoremstyle{definition}
\newtheorem{remark}[theorem]{Remark}
\newtheorem{definition}[theorem]{Definition}
\numberwithin{equation}{section}

\numberwithin{equation}{section}

\numberwithin{equation}{section}

\begin{document}

\arraycolsep=1pt

\title{\Large\bf A Representation of Matrix-Valued Harmonic Functions by the Poisson Integral of Non-commutative BMO Functions
 \footnotetext{\hspace{-0.35cm} {\it 2010 Mathematics Subject Classification}.
{Primary 46L52; Secondary 42B30, 42B35, 42C40.}
\endgraf{\it Key words and phrases.}  Non-commutative $L_p$-space, BMO space, Hardy space, Carleson measure, Dirichlet problem.
\endgraf G. Hong was supported by National Natural Science Foundation of China (No. 12071355, No. 12325105) and Fundamental Research Funds for the Central Universities (No. 2042022kf1185). W. Wang was supported by
 China Postdoctoral Science Foundation (No. 2024M754159),  Postdoctoral Fellowship Program of CPSF (No. GZB20230961) and
  Heilongjiang Provincial Postdoctoral Science Foundation (No. LBH-Z23174).
 C, Chen was supported by Fundamental Research Funds for the Central Universities, Sun Yat-Sen University (No. 23qnpy50).
}}
\author{Cheng Chen, Guixiang Hong and Wenhua Wang}
\date{  }
\maketitle

\vspace{-0.8cm}

\begin{center}
\begin{minipage}{13cm}\small
{\noindent{\bf Abstract} \
  In this paper, the authors study
the matrix-valued harmonic functions
and characterize them by the Poisson integral of functions in non-commutative
 BMO (bounded mean oscillation) spaces. This provides a very satisfactory non-commutative analogue of the beautiful result due to
Fabes, Johnson and Neri [Indiana Univ. Math. J. $\mathbf{25}$ (1976) 159-170; MR0394172].
}
\end{minipage}
\end{center}

\section{Introduction}
\hskip\parindent
In the past few years, the study of harmonic extension of a function
 is one of the essential topic in Harmonic Analysis
and {PDE}s. That is, given a domain $\Omega \subseteq \mathbb{R}^{n+1}$ and a function $f$   on the boundary $\partial \Omega$,  we set out to find a function $u$ which is harmonic on $\Omega$ but continuous on $\partial \Omega$ such that the restriction of $u$ to $\partial \Omega$ coincides with $f$. A typical case is that $\Omega=\mathbb{R}^{n+1}_{+}$ and $\partial \Omega \cong \mathbb{R}^{n}$. In the course of showing the duality of classical Hardy and BMO spaces,
Fefferman and Stein \cite{fs72} proved that
the harmonic extension $u(x,\,t) = e^{-t\sqrt{-\Delta}}(f)(x)$ satisfies the following
Carleson condition:
\begin{align}\label{e1.1}
\sup _{y\in\rn,\,r>0}\lf(\frac{1}{\lf|B(y,\,r)\r|}
\int_{B(y,\,r)}\int_{0}^{r}\lf|t\nabla u(x,\,t)\r|^2\,\frac{dtdx}{t}\r)^{1/2}<\infty,
\end{align}
if and only if $f$ is a BMO function,
where $\nabla=(\frac{\partial}{\partial{x_1}},\,\ldots,\,
 \frac{\partial}{\partial{x_n}},\,\frac{\partial}{\partial{t}})$ and
  $\Delta=\sum_{i=1}^n\frac{\partial^2}{\partial x_i^2}+\frac{\partial^2}{\partial t^2}$.
Strengthening this result, Fabes, Johnson and Neri \cite{fjn75} showed that condition \eqref{e1.1} actually characterizes all the harmonic functions on $\rr^{n+1}_+$ whose traces are in $\mathrm{BMO}(\rn)$.
 The study of this topic has been widely extended
to different settings, for instance, for different equations or systems such as degenerate elliptic equations and systems, elliptic equations and
systems with complex coefficients, Schr\"{o}dinger equations, and for domains other than $\rn$ such as Lipschitz domains (see e.g. \cite{dkp11,dyz14,fn75,fn80,g14,jl22}).

On the other hand, matrix-valued harmonic analysis on $\rn$ plays an important role in many areas of mathematics. For example, it has applications in the prediction theory and rational approximation (see e.g. \cite{ds23,p00}).
More generally, operator-valued harmonic analysis has been developed rapidly in recent years. For instance, motivated by the theory of non-commutative martingale inequalities
 (see e.g. \cite{j02,jx03,jx07,jx08,lp91,px97}) and the Littlewood-Paley-Stein theory of quantum Markov semigroups (see e.g. \cite{jlx06,jm10,jm12}), in 2006, Mei \cite{m07} studied systematically the theory of non-commutative Hardy spaces and BMO spaces on the Euclidean space $\rn$. In particular, he provided a quite satisfactory non-commutative analogue of the famous Fefferman-Stein duality between Hardy space and BMO space (see \cite[Lemma 1.4 and Corollary 2.6]{m07}), which solved an open question in matrix-valued harmonic analysis arising from prediction theory.
For more information about non-commutative Hardy spaces and BMO spaces, we refer the readers to \cite{hlm20,hww22,hx21,hy13,w23,xx18,xx19,xxx16}.

Inspired by these results,
 it is natural and interesting to ask whether there holds a non-commutative version of the beautiful result of Fabes {\it et al.} \cite{fjn75}. To be more precise, whether the Carleson condition \eqref{e1.1} can characterize all matrix-valued harmonic functions
$u(x,\,t)$ on $\rr^{n+1}_+$ with boundary value in non-commutative BMO spaces?

{In this paper we shall give an affirmative answer.
To be more precise, let $B(\ell_2)$ stand for the matrix algebra of bounded
linear operators on $\ell_2$. We consider the algebra formed by essentially
bounded functions $f:\rn\rightarrow B(\ell_2)$.
We prove that matrix-valued harmonic function $u:\rr_+^{n+1}\rightarrow B(\ell_2)$ satisfies a
Carleson condition similar to \eqref{e1.1} (will be stated explicitly later)
if and only if $u$ is the Poisson integral of some function in non-commutative BMO spaces.}

This result also holds for general operator-valued harmonic function, that is, replacing the
matrix algebra $B(\ell_2)$ by an arbitrary semifinite von Neumann algebra $\cm$.  In what follows, we will prove our result in this general framework.

This paper is organized as follows.

{In Section \ref{s2}, we introduce the notions of being harmonic with respect to different topologies,
and then show the equivalence of these {\it a priori} different notions. For more details, we refer the readers to  Theorem \ref{t4.1} below.
There needs to be emphasized that, in the process of proving our main result---Theorem \ref{t3.1}, this equivalence plays a crucial role for the desired
estimates.}

In Section \ref{s3}, we first recall some preliminaries
concerning the non-commutative $L_p$-spaces as well as the spaces of the operator-valued harmonic functions satisfying the Carleson measure condition. Then we prove the main result, that is, characterize the operator-valued harmonic functions by the Poisson integral of non-commutative
 BMO functions. It is worth to be mentioned that, due to the noncommutativity, some methods in
the proof of \cite[Theorem 1]{fjn75}  do not apply directly in the present setting.
To overcome this difficulty, we utilize the equivalence of the different notions of being harmonic with respect to the different topologies.

Finally, we make some conventions on notations.

Set $\nn:=\{1,\, 2,\,\ldots\}$ and $\zz_+:=\{0\}\cup\nn$.
Throughout the whole paper, we denote by $C$ a positive
constant which is independent of the main parameters, but it may
vary from line to line. We also use $C_{(\alpha,\,\beta,\,\ldots)}$ to denote a positive constant
depending on the indicated parameters $\alpha$, $\beta$, $\cdots$.  For any $x\in\rn$, $r\in(0,\,\infty)$, let $B(x,\,r):=\{y\in\rn:|x-y|<r\}$.
The symbol $D\ls F$ means that $D\le
CF$ for some constant $C\in \mathbb{R}_{+}$.  If $D\ls F$ and $F\ls D$, we then write $D\sim F$.
For a set $E \subset\rn$, let $E^\complement:=\rn\setminus E$ and we denote by $\chi_E$ its characteristic
function. Let $C^2(\rr_+^{n+1})$ denote the set of
 all the twice continuously differentiable functions on $\rr_+^{n+1}$.
 For a Banach space $X$, let ${X}_*$ denote a {predual space} of $X$ and ${X}^*$ be the {dual space} of ${X}$.



\section{Operator-valued harmonic functions \label{s2}}
\bigskip

In this section, we introduce the different definitions of being harmonic of an operator-valued function, and then show all of them are equivalent.

Let $C^2(\rr_+^{n+1},{X})$ denote the set of all the twice continuously differentiable functions from $\rr_+^{n+1}$ to a Banach space ${X}$, where the derivatives are taken with respect to the norm. Without causing any confusion, for $f\in C^2(\rr_+^{n+1},{X})$, we still use  $\Delta f=\sum_{i=1}^n\frac{\partial^2f}{\partial x_i^2}+\frac{\partial^2f}{\partial t^2}$ to denote the associated Laplacian operator. For the vector-valued functions, unlike in the scalar-valued case, the notion of being harmonic varies when the underlying Banach
space is endowed with different topologies.

Let $\cm$ be a semifinite von Neumann algebra, and $H$ denote a Hilbert space on which $\cm$ acts. The standard inner product on $H$ will be denoted by $\langle \cdot,\,\cdot\rangle_H$. The dual (resp. predual) space of $\mathcal M$ is denoted by $\cm^*$ (resp. $\cm_*$). The pairing of $\cm$ and $\cm_*$ will be denoted by $\langle\cdot,\,\cdot\rangle$.

\begin{definition}\label{d2.1}
 Let $u:\rr_+^{n+1}\rightarrow \cm$ be an operator-valued function.
We say
\begin{enumerate}
\item[\rm{(i)}]$u$ is {\it strongly harmonic} if $u\in C^2(\rr_+^{n+1},\,\cm)$ and
$\Delta u=0$;
\item[\rm{(ii)}] $u$ is {\it weakly harmonic} if $\varphi(u)$ is harmonic for any $\varphi\in\cm^*$;
 \item[\rm{(iii)}] $u$ is {\it weak-$^*$ harmonic}, if $\langle u,\,a\rangle$ is harmonic for any $a\in \cm_*$;
  \item[\rm{(iv)}]$u$ is {\it wo-harmonic} if $\langle \xi,\,u\eta\rangle_H$ is harmonic for any $\xi, \eta\in H$;
 \item[\rm{(v)}]$u$ is {\it so-harmonic} if $ u\xi \in C^2(\rr_+^{n+1},\,H)$ and $\Delta (u\xi)=0$ for any $\xi\in H$;
 \item[\rm{(vi)}] $u$ is {\it $\sigma$-so-harmonic} if $\{u\xi_j\}_j \in C^2(\rr_+^{n+1},\,\ell_2(H))$ and $\Delta (u\{\xi_j\}_j)=0$ for any $\{\xi_j\}_j\in \ell_2(H)$;
 \item[\rm{(vii)}] $u$ is {\it $*$-so-harmonic} if both $u$ and $u^*$ are so-harmonic;
 \item[\rm{(viii)}] $u$ is {\it $\sigma$-$*$-so-harmonic} if both $u$ and $u^*$ are $\sigma$-so-harmonic.
\end{enumerate}
\end{definition}

In the following theorem, we show the equivalence of being harmonic with respect to the different topologies.
\begin{theorem}\label{t4.1}
Let $u:\rr_+^{n+1}\rightarrow\cm$ be an operator-valued function.
Then the following statements are equivalent:
\begin{enumerate}
\item[\rm{(i)}]$u$ is strongly harmonic;
\item[\rm{(ii)}] $u$ is weakly harmonic;
 \item[\rm{(iii)}]$u$ is weak-$^*$ harmonic;
 \item[\rm{(iv)}]$u$ is wo-harmonic;
 \item[\rm{(v)}]$u$ is so-harmonic;
 \item[\rm{(vi)}]$u$ is $\sigma$-so-harmonic;
 \item[\rm{(vii)}] $u$ is $*$-so-harmonic;
 \item[\rm{(viii)}] $u$ is $\sigma$-$*$-so-harmonic.
\end{enumerate}
\end{theorem}
To prove Theorem \ref{t4.1}, we need some technical lemmas. The first one is the Hille theorem, see e.g. \cite{du77}.
\begin{lemma}\label{l2.tt}
Let $(\Omega,\,d\mu)$ be a measure space, $X$, $Y$ be two Banach spaces,  $T:X\rightarrow Y$ be a closed linear operator defined inside $X$ and having values in $Y$. If $f:\Omega\rightarrow X$ and $Tf$ are Bochner integrable with respect to $\mu$, then we have
$$T\left(\int_E f(x)\,d\mu(x) \right)=\int_E T(f(x))\,d\mu(x),$$
for any measurable set $E\subseteq\Omega$.
\end{lemma}

\begin{lemma}\label{l4.1}
If $u: \mathbb{R}_+^{n+1}\rightarrow\cm$ is weakly continuous, then $u$ is Bochner integrable over any compact subset $K\subseteq \mathbb{R}_+^{n+1}$.
\end{lemma}
To prove this, we borrow some ideas from \cite{a16}.
\begin{proof}
Firstly, we show that $\|u(\cdot,\,\cdot)\|:\mathbb{R}_+^{n+1}\rightarrow\mathbb{R}_+$ is measurable. Let $D\subseteq \mathbb{R}_+^{n+1}$ be a countable dense subset. Then $u(\mathbb{R}_+^{n+1})$ is contained in the weak closure of $u(D)$. By the  Hahn-Banach theorem, the norm closure and weak closure of any convex subsets of $\cm$ coincide. Therefore, $$u(\mathbb{R}_+^{n+1})\subseteq \overline{u(D)}^{w}\subseteq \cm_{0}:=\overline{\text{span}\{u(D)\}}^{\|\cdot\|}.$$
Let $\{a_{j}\}_{j}\subseteq \cm_{0}\setminus\{\mathbf{0}\}$ be dense. Choose the sequence $\{\varphi_{j}\}_{j}\subseteq \cm^{*}$ such that
$\|\varphi_{j}\|=\varphi_{j}(a_{j}/\|a_{j}\|)=1,$
 then $\|u(x,\,t)\|=\text{sup}_{j}|\varphi_{j}(u(x,\,t))|$ for any $(x,\,t)\in \mathbb{R}_+^{n+1}$. Indeed, for any $j\in\nn$,
 $$
 |\varphi_{j}(u(x,\,t))|\leq \|\varphi_{j}\|\|u(x,\,t)\|= \|u(x,\,t)\|.
 $$
 For any $\epsilon>0$, there exists $j\in\nn$ such that $$\|u(x,\,t)-a_{j}\|<\epsilon/2.$$
In this case, $$|\varphi_{j}(u(x,\,t))|\geq |\varphi_{j}(a_{j})|-\epsilon/2=\|a_{j}\|-\epsilon/2\geq \|u(x,\,t)\|-\epsilon.$$ Thus, by the assumption that $\varphi_{j}(u)$ is continuous for any $j\in\nn$, we know that $\|u(\cdot,\,\cdot)\|$ is measurable on $\mathbb{R}_+^{n+1}$.

In the next, we show that $u$ is Bochner measurable. Note that $\|u(\cdot,\,\cdot)-a_{j}\|$ is measurable for any $j\in\nn$. We define
$$E_{j}^{m}:=\lf\{(x,\,t)\in \mathbb{R}_+^{n+1}: \|u(x,\,t)-a_{j}\|<1/m\r\},$$
then $E_{j}^{m}$ is measurable, and $\cup_{j\geq 1}E_{j}^{m}=\mathbb{R}_+^{n+1}$ for all $m$. Define $g_{m}:\ \mathbb{R}_+^{n+1}\rightarrow \cm$ by $$g_{m}(x,\,t)=a_{j} \ \ \mathrm{with} \ \  j=\inf\lf\{i:  E_{i}^{m}\ni(x,\,t)\r\}.$$
Then $\|u(x,\,t)-g_{m}(x,\,t)\|<1/m$ for all $(x,\,t)\in \mathbb{R}_+^{n+1}$. That is, $u$ can be uniformly approximated by countably valued function. For any compact subset $K\subseteq \mathbb{R}_+^{n+1}$ and $\epsilon>0$, there exist simple functions $g_{m,K,\epsilon}:\ K\rightarrow \cm$ and measurable subset $E_{K,\epsilon}\subseteq K$ such that
 $$
\lf|E_{K,\epsilon}\r|<\epsilon \ \ \mathrm{and} \ \ \lf\|u(x,\,t)-g_{m,K,\epsilon}(x,\,t)\r\|<1/m, \ \  \mathrm{for\ all} \ (x,\,t)\in K\setminus E_{K,\epsilon}.
$$
 Let $\{\epsilon_{j}\}_{j}$ be a sequence decreasing to $0$, define $g_{m,K}:\ K\rightarrow \cm$ by
\begin{eqnarray*}
g_{m, K}(x,\,t):=
\lf\{ \begin{array}{ll}
&g_{m,K,\epsilon_{j}}(x,\,t), \ \  \mathrm{if} \ j:=\inf\{i: (x,\,t)\in K\setminus E_{K,\epsilon_{i}}\}\leq m, \\
&\ \ \\
&\ \ \ \ 0 \ \ \ \ \ \ \ \ \ \ \ \ \  \, \mathrm{if} \ j:=\text{inf}\{i: (x,\,t)\in K\setminus E_{K,\epsilon_{i}}\}> m.
\end{array}\r.
\end{eqnarray*}
Then $\{g_{m,K}\}_{m}$ are simple functions on $K$ and converges in norm to $u$ almost everywhere in $K$. Indeed, $$\lf|\bigcap_{j}E_{K,\epsilon_{j}}\r|\leq \lf|E_{K,\epsilon_{j}}\r|<\epsilon_{j},$$ That is,
$$
\lf|\bigcap_{j}E_{K,\epsilon_{j}}\r|=0.
$$
For all $x\in K\setminus \cap _{j}E_{K,\epsilon_{j}}$, $\{j\in\nn: x\in K\setminus E_{K,\epsilon_{j}}\}$ is not empty. For all $m\geq j:=\inf\{i: x\in K\setminus E_{K,\epsilon_{i}}\}$, $g_{m,K}(x)=g_{m, K,\epsilon_{j}}(x)$, so $g_{m, K}$ converges in norm to $u$. Let $\{K_{j}\}_{j}$ be an increasing sequence consisting of compact subsets of $\mathbb{R}_+^{n+1}$ such that $\cup_{j}K_{j}=\mathbb{R}_+^{n+1}$, define $u_{m}:\ \mathbb{R}_+^{n+1}\rightarrow \cm$ by
\begin{eqnarray*}
u_{m}(x,\,t):=
\lf\{ \begin{array}{ll}
&g_{m, K_{j}}(x,\,t), \ \  \mathrm{if} \ j:=\inf\{i: (x,\,t)\in K_{i}\}\leq m, \\
&\ \ \\
&\ \ \ \ 0 \ \ \ \ \ \ \ \ \ \ \ \ \ \mathrm{if} \ j:=\inf\{i: (x,\,t)\in K_{i}\}>m.
\end{array}\r.
\end{eqnarray*}
Then $\{u_{m}\}_{m}$ are simple functions on $\mathbb{R}_+^{n+1}$. For all $(x,\,t)\in \mathbb{R}_+^{n+1}$ except for a null set, when $m$ is big enough, we always have $u_{m}(x,\,t)=g_{m,K_{j}}(x,\,t)$, converging in norm to $u(x,\,t)$.

Finally, we show that, for all compact subsets $K\subseteq \mathbb{R}_+^{n+1}$,
$$
\int_{K}\|u(x,\,t)\|\,dxdt<+\infty.
$$
It suffices to show that $\|u(\cdot,\,\cdot)\|$ is bounded on $K$. Notice the fact that $\{u(x,\,t)\}_{(x,\,t)\in K}$ is a family of linear functionals on $\cm^{*}$ by $$u(x,\,t)(\varphi):=\varphi(u(x,\,t))$$ and that $\varphi(u)$ is continuous and thus bounded over $K$, then the boundedness of $u$ over $K$ follows from the Banach-Steinhaus theorem.
\end{proof}

Now we prove the main result of this section.
\begin{proof}[Proof of Theorem \ref{t4.1}]
We prove Theorem \ref{t4.1} in two steps.

\textbf{Step 1.} In this step, we show the equivalence of (i), (ii) and (iii).
It is easy to obtain (i)$\Longrightarrow$ (ii) and (ii) $\Longrightarrow$ (iii). Therefore, to complete the proof of \textbf{Step 1}, we only need to prove that (ii) $\Longrightarrow$ (i) and (iii) $\Longrightarrow$ (ii).

(ii)$\Longrightarrow$ (i): Let $u$ be a weakly harmonic function. Now we show that $u$ is also strongly harmonic. For any $\overline{x_0}:=(x_{0},\,t_0)\in \mathbb{R}_+^{n+1}$, there exists a closed ball $B:=\overline{B(\overline{x_0},\ r)}\subseteq \mathbb{R}_+^{n+1}$. Since $u$ is weakly harmonic, it follows from Lemma \ref{l4.1} that $u$ is Bochner integrable over $\partial B$,  and thus, for all $\overline{x}:=(x,\,t)\in \mathbb{R}_+^{n+1}$ with $|\overline{x}-\overline{x_{0}}|<r$, $P(\overline{x},\,\zeta):=\frac{r^{2}-|\overline{x}-\overline{x_{0}}|^{2}}{|\overline{x}-\zeta|^{n+1}}u(\zeta)$ is also a Bochner integrable function of $\zeta$ over $\partial B$. Therefore, for all $\varphi\in \cm^*$, by the Hille theorem (see Lemma \ref{l2.tt}), we have
$$\varphi\big(\int_{\partial B}\frac{r^{2}-|\overline{x}-\overline{x_{0}}|^{2}}{|\overline{x}-\zeta|^{n+1}}u(\zeta)
\frac{d\mu(\zeta)}{\mu(\partial B)}\big)=\int_{\partial B}\frac{r^{2}-|\overline{x}-\overline{x_{0}}|^{2}}{|\overline{x}-\zeta|^{n+1}}
\varphi(u(\zeta))\frac{d\mu(\zeta)}{\mu(\partial B)}.$$
Since $\varphi(u)$ is harmonic, by the Poisson formula of $\varphi(u(\overline{x}))$, we have
$$\varphi(u(\overline{x}))=\varphi\big(\int_{\partial B}\frac{r^{2}-|\overline{x}-\overline{x_{0}}|^{2}}{|\overline{x}-\zeta|^{n+1}}u(\zeta)
\frac{d\mu(\zeta)}{\mu(\partial B)}\big)$$
for all $\varphi\in \cm^*$. By the Hahn-Banach theorem,  this yields that
$$u(\overline{x})=\int_{\partial B}\frac{r^{2}-|\overline{x}-\overline{x_{0}}|^{2}}{|\overline{x}-\zeta|^{n+1}}
u(\zeta)\frac{d\mu(\zeta)}{\mu(\partial B)}.$$
It is obvious that the function $P(\overline{x},\ \zeta)=\frac{r^{2}-|\overline{x}-\overline{x_{0}}|^{2}}{|\overline{x}-\zeta|^{n+1}}$ can be written as a series $\sum_{\alpha}c_{\alpha}(\zeta)(\overline{x}-\overline{x_{0}})^{\alpha}$, which converges uniformly in some neighborhood of $\overline{x_{0}}$.
Interchanging the integral and the sum, we get
$$u(\overline{x})=\sum_{\alpha}\lf(\int_{\partial B}c_{\alpha}(\zeta)u(\zeta)\frac{d\mu(\zeta)}{\mu(\partial B)}\r)(\overline{x}-\overline{x_{0}})^{\alpha}.$$
This holds when $\overline{x_{0}}$ varies arbitrarily in $\mathbb{R}_+^{n+1}$ and $\overline{x}$ keeps close enough to $\overline{x_{0}}$, so $u\in C^{2}(\mathbb{R}_+^{n+1},\ \cm)$. Furthermore,
$$\varphi(\Delta (u))(\overline{x})=\sum_{\alpha}\varphi(\int_{\partial B}c_{\alpha}(\zeta)u(\zeta)\frac{d\mu(\zeta)}{\mu(\partial B)})\Delta(\overline{x}-\overline{x_{0}})^{\alpha}=\Delta \varphi (u)(\overline{x})=0,$$
for all $\varphi\in \cm^*$. By the Hahn-Banach theorem, we have $\Delta u=0$.
Thus $u$ is strongly harmonic.

(iii)$\Longrightarrow$(ii): Let $u$ be a weak-$^*$ harmonic function. Now we need to show that, for any $\varphi\in \cm^*$,  $\varphi(u(\cdot)):\mathbb{R}_+^{n+1}\rightarrow\mathbb{C}$ is harmonic. It is easy to see that there exists $\{\varphi_{j}\}_{j}\subseteq \cm_*$ such that $\varphi_{j}(b)$ converges to $\varphi(b)$ for all $b\in \cm$. By the Banach-Steinhaus theorem, we know that $\{\varphi_{j}\}_{j}$ is uniformly bounded. As in Lemma \ref{l4.1}, we find that $u$ is also bounded over any compact subset $K\subseteq \mathbb{R}_+^{n+1}$. From the fact that $\varphi_{j}(u)$ is harmonic for any $j\in\nn$, and the Poisson formula, we deduce that, for any $\overline{x}\in\mathbb{R}_+^{n+1}$,
\begin{align}\label{e4.8}
\varphi_{j}(u(\overline{x}))=\int_{\partial B}P(\overline{x},\ \zeta)\varphi_{j}(u(\zeta))\frac{d\mu(\zeta)}{\mu(\partial B)}.
\end{align}
Notice that, for any $\overline{x},\,\overline{x}'\in\mathbb{R}_+^{n+1}$,
$$\lf|\varphi_{j}(u(\overline{x}))-\varphi_{j}(u(\overline{x}'))\r|\leq \int_{\partial B}\lf|P(\overline{x},\ \zeta)-P(\overline{x}',\ \zeta)\r||\varphi_{j}(u(\zeta))|\frac{d\mu(\zeta)}{\mu(\partial B)},$$
which implies the equicontinuity of the sequence of functions $\{\varphi_{j}(u)\}_{j}$. By the Ascoli theorem, we know that $\varphi_{j}(u)$ converges uniformly to $\varphi(u)$ on any compact subsets $K\subseteq \mathbb{R}_+^{n+1}$, as $j\rightarrow\fz$. Taking the limit in \eqref{e4.8} as $j\rightarrow\fz$,
we get the Poisson formula of $\varphi(u)$. That is, for any $\overline{x}\in\mathbb{R}_+^{n+1}$,
\begin{align*}
\varphi(u(\overline{x}))=\int_{\partial B}P(\overline{x},\ \zeta)\varphi(u(\zeta))\frac{d\mu(\zeta)}{\mu(\partial B)}.
\end{align*}
Therefore, this implies that $\varphi(u)$ is continuous and harmonic.

\textbf{Step 2.} In this step, we
 show that the equivalence of (iv), (v), (vi) (vii) and (viii). It is obvious that (i) implies (v), (vi), (vii) and (viii) and each of (v), (vi), (vii) and (viii) also  implies (iv) since the topology induced by the norm in $B(H)$ is finest, while the weak operator topology is the coarsest. To finish the proof of \textbf{Step 2}, it suffices to prove that (iv)$\Longrightarrow$(iii).

(iv)$\Longrightarrow$(iii): Let $\langle\xi,\, u(\cdot)\eta\rangle_{H}$ be a harmonic function on $\mathbb{R}_+^{n+1}$ for any $\xi,\,\eta\in H$. Next we only need to show that
for any $\{\xi_{j}\}_{j},\ \{\eta_{j}\}_{j}\subseteq H$ with $\sum_{j}\|\xi_{j}\|_{H}\|\eta_{j}\|_{H}<+\infty$,
$$\sum_{j\in\nn}\langle\xi_{j},u(\cdot)\eta_{j}\rangle_H$$ is a harmonic function on $\mathbb{R}_+^{n+1}$.

For any compact subset $K\subseteq \mathbb{R}_+^{n+1}$ and $\eta\in H$, $\{u(\overline{x})\eta\}_{\overline{x}\in K}$ is a family of bounded linear functionals on $H$. In fact, for all $\xi\in H$, $u(\overline{x})(\eta)(\xi)=\langle\xi, u(\overline{x})\eta\rangle_H$ is harmonic and thus bounded over $K$. It follows from the Banach-Steinhaus theorem that $\{u(\overline{x})\eta\}_{\overline{x}\in K}$ is bounded in $H$. A similar argument yields that $\{u(\overline{x})\}_{\overline{x}\in K}$ is bounded in $\cm$. For all $\{\xi_{j}\}_{j},\ \{\eta_{j}\}_{j}\subseteq H$ with $\sum_{j}\|\xi_{j}\|_{H}\|\eta_{j}\|_{H}<+\infty$ and all $J\in \mathbb{N}$, we have
$$\lf|\sum_{j\geq J}\lf\langle\xi_{j},u(\overline{x})\eta_{j}\r\rangle_H\r|\leq \sum_{j\geq J}\lf|\lf\langle\xi_{j},u(\overline{x})\eta_{j}\r\rangle_H\r|\leq \text{sup}_{\overline{x}\in K}\|u(\overline{x})\|\sum_{j\geq J}\|\xi_{j}\|_{H}\|\eta_{j}\|_{H}.$$
This means that $\sum_{j}\langle\xi_{j},u(\cdot)\eta_{j}\rangle_H$ converges uniformly on all compact subsets $K\subseteq \mathbb{R}_+^{n+1}$. Notice that $\sum_{j\leq J}\langle\xi_{j},u(\cdot)\eta_{j}\rangle_H$ is harmonic on $\mathbb{R}^{n+1}_+$ for all $J$, thus $$\sum_{j}\langle\xi_{j},u(\cdot)\eta_{j}\rangle_H$$ is also harmonic on $\mathbb{R}^{n+1}_+$. Notice that each $\varphi \in B(H)_{*}$ is of the form $\sum_{j}\xi_{j}\otimes \eta_{j}$ for some $\{\xi_{j}\}_{j},\ \{\eta_{j}\}_{j}\subseteq H$ with $\sum_{j}\|\xi_{j}\|_{H}\|\eta_{j}\|_{H}<+\infty$, therefore we complete the proof of Theorem \ref{t4.1}.
\end{proof}

By Theorem \ref{t4.1}, we are able to give the definition of operator-valued harmonic functions.
\begin{definition}\label{d2.5y}
Let $u:\rr_+^{n+1}\rightarrow \cm$ be an operator-valued function.
If $u$ satisfies any of Definitions \ref{d2.1} (i)--(viii), then we say that $u$ is an operator-valued harmonic function on $\rr_+^{n+1}$.
\end{definition}

The following conclusions follow from Lemma \ref{l2.tt}, Theorem \ref{t4.1} and its proof. We collect them here for further applications in the next section.

\begin{corollary}\label{c4.9}
Let $u:\rr_+^{n+1}\rightarrow \cm$ be an operator-valued harmonic function. Then
\begin{enumerate}
\item[\rm{(i)}]
for any $\overline{x}\in\rr_+^{n+1}$ and $r>0$, we have
$$u(\overline{x})=\frac{1}{|B(\overline{x},\,r)|}
\int_{B(\overline{x},\,r)}u(\overline{y})\,d
\overline{y};$$
\item[\rm{(ii)}]
for any $a\in \cm_*$ and $i\in[1,\,n]\cap\nn$, we have
$$\frac{\partial}{\partial x_i}\langle u,\,a\rangle
=\left\langle\frac{\partial}{\partial x_i} u,\,a\right\rangle \ \ \ \mathrm{and} \ \
\frac{\partial}{\partial t}\langle u,\,a\rangle
=\left\langle\frac{\partial}{\partial t} u,\,a\right\rangle.$$
\end{enumerate}
\end{corollary}
\begin{proof}

(i) Since $u$ is weakly harmonic, it follows from Lemma \ref{l4.1} that $u$ is Bochner integrable over $B(\overline{x},\,r)\subset\rr_+^{n+1}$, $\overline{x}\in\rr_+^{n+1}$ and $r>0$.
Therefore, for all $\varphi\in \cm^*$, Lemma \ref{l2.tt} and the harmonicity of $\varphi(u)$ imply
$$\varphi(u(\overline{x}))=\frac{1}{|B(\overline{x},\,r)|}
\int_{B(\overline{x},\,r)}\varphi(u(\overline{y}))\,d
\overline{y}=\varphi\lf(\frac{1}{|B(\overline{x},\,r)|}
\int_{B(\overline{x},\,r)}u(\overline{y})\,d
\overline{y}\r).$$
Therefore, from the arbitrariness of $\varphi\in\cm^*$, we conclude the desired identity: for any $\overline{x}\in\rr_+^{n+1}$ and $r>0$,
$$u(\overline{x})=\frac{1}{|B(\overline{x},\,r)|}
\int_{B(\overline{x},\,r)}u(\overline{y})\,d
\overline{y}.$$

(ii) Since $u$ is strongly harmonic, $\frac{u(\overline{x}+he_i)-u(\overline{x})}{h}$ converges to $\frac{\partial}{\partial x_i} u(\overline{x})$ in $\cm$ as $h\rightarrow0$, where $e_i\in\rr^{n}$ is $i$-th unit coordinate vector. Thus for any $a\in \cm_{*}$ and $i\in[1,\,n]\cap\nn$, we have
\begin{align*}
\frac{\partial}{\partial x_i}\langle u(\overline{x}),\,a\rangle
=&\lim_{h\rightarrow0}\frac{\langle u(\overline{x}+he_i),\,a\rangle-\langle u(\overline{x}),\,a\rangle}{h}\\
=&\lim_{h\rightarrow0}\lf\langle \frac{u(\overline{x}+he_i)-u(\overline{x})}{h},\,a\r\rangle\\
=&\lf\langle\frac{\partial}{\partial x_i} u(\overline{x}),\,a\r\rangle.
\end{align*}
The identity involving $\frac{\partial}{\partial t}$ is dealt with in the same way.
\end{proof}

\section{The main result and its proof \label{s3}}
\bigskip
In this section, we will prove the main result, that is, the characterization of operator-valued harmonic functions in terms of non-commutative $\mathcal{BMO}$ spaces, see Theorem \ref{t3.1} below.

Firstly, let us recall the definition of non-commutative $L_p$-spaces. Let
$\cm$ be a von Neumann algebra equipped with a normal semifinite faithful trace $\tau$.
Let $S_{\cm}^+$ be the set of all positive $x\in\cm$ such that $\tau(s(x))<\infty,$
where $s(x)$ denotes the support of $x$, that is, the least projection $e\in\cm$ such that $exe=x.$
Let $S_\cm$ be the linear span of $S_{\cm}^+$. For any $p\in(0,\,\infty)$, we define
$$\|x\|_{L_p(\mathcal M)}:=(\tau|x|^p)^{1/p}, \ \ \ x\in S_\cm,$$
where $|x|:=(x^*x)^{1/2}$. The usual non-commutative $L_p$-space, $L_p(\cm)$, associated with $(\cm,\,\tau)$, is
the completion of $(S_\cm, \|\cdot\|_{L_p(\mathcal M)})$. For $p=\infty$, we set $L_{\infty}(\cm)=\cm$ equipped
with the operator norm $\|\cdot\|$.

\smallskip

Next we introduce the spaces of harmonic functions. Let $C^{\mathrm{H}}(\rr_+^{n+1},\,\cm)$ denote the set of all the harmonic functions on $\rr_+^{n+1}$ with values in $\cm$.

\begin{definition}
The {column space of harmonic functions} is defined as
\begin{eqnarray*}
\mathcal{HMO}^{c}(\rr^{n+1}_+,\,\cm):=\lf\{u\in C^{\rm{H}}(\rr^{n+1}_+,\,\cm):\|u\|_{\mathcal{HMO}^{c}(\rr^{n+1}_+,\,\cm)}<\infty\r\},
\end{eqnarray*}
where
$$\|u\|_{\mathcal{HMO}^{c}(\rr^{n+1}_+,\,\cm)}:=\sup_{\mathrm{ball}\, B\subset\rn}\lf\|\lf(\frac{1}{|B|}
\int_{B}\int_{0}^{r_B}\lf|t\nabla u(x,\,t)\r|^2\,\frac{dxdt}{t}\r)^{1/2}\r\|$$
 with $r_B$ denoting the radius of ball $B$,
and
$$|\nabla u(\cdot,\,\cdot)|^2=\lf(\frac{\partial u}{\partial x_1}\r)^*\frac{\partial u}{\partial x_1}+\cdots+\lf(\frac{\partial u}{\partial x_n}\r)^*\frac{\partial u}{\partial x_n}+\lf(\frac{\partial u}{\partial t}\r)^*\frac{\partial u}{\partial t}.$$
Similarly, we define the {row space of harmonic functions} $\mathcal{HMO}^{r}(\rr^{n+1}_+,\,\cm)$ as the space of $u$ such that $u^{\ast}\in \mathcal{HMO}^{c}(\rr^{n+1}_+,\,\cm)$ with norm
$\|u\|_{\mathcal{HMO}^{r}(\rr^{n+1}_+,\,\cm)}
:=\|u^{\ast}\|_{\mathcal{HMO}^{c}(\rr^{n+1}_+,\,\cm)}$.
\end{definition}
\begin{remark}
When it comes back to the commutative setting, i.e.,
 $\cm:=\mathbb{C}$, the operator-valued harmonic function spaces $\mathcal{HMO}^c(\rr_+^{n+1},\,\cm)$ and $\mathcal{HMO}^r(\rr_+^{n+1},\,\cm)$
  are reduced to the classical (commutative) harmonic function space $\mathrm{HMO}(\rr_+^{n+1})$ studied by Fabes, Johnson and Neri \cite{fjn75}.
\end{remark}

Recall that the column operator-valued BMO space $\mathcal{BMO}^{c}(\rn,\,\cm)$ is defined as a subspace of $L_{\infty}(\cm;\,L^\mathrm{c}_2(\rn,\,\frac{dx}{1+|x|^{n+1}}))$ with
$$\|f\|_{\mathcal{BMO}^{c}(\rn,\,\cm)}:=\sup_{\mathrm{ball}\, B\subset\rn}\lf\|\lf[\frac{1}{|B|}
\int_{B}\lf|f(x)-\frac{1}{|B|}\int_Bf(y)\,dy\r|^2\,dx\r]^{1/2}\r\|.$$
Then {the operator-valued row BMO space} $\mathcal{BMO}^{r}(\rn,\,\cm)$ is defined as the space of $f$ such that $f^{\ast}\in \mathcal{BMO}^{c}(\rn,\,\cm)$ with norm
$\|f\|_{\mathcal{BMO}^{r}(\rn,\,\cm)}:=\|f^{\ast}\|_{\mathcal{BMO}^{c}(\rn,\,\cm)}$.  We refer to Mei's seminal work \cite{m07} for more properties of these spaces.

\smallskip

The main result of this subsection is as follows.
\begin{theorem}\label{t3.1}
We have the following conclusions:
\begin{enumerate}
\item[\rm{(i)}]Let $u\in\mathcal{HMO}^c(\rr_+^{n+1},\,\cm)$. Then there exists $f\in\mathcal{BMO}^c(\rn,\,\cm)$ such that, for any $(x,\,t)\in\rr_+^{n+1}$, $u(x,\,t)=f*P_t(x)$. Moreover, there exists a positive constant $C$ such that
    $$\|f\|_{\mathcal{BMO}^c(\rn,\,\cm)}\leq C\|u\|_{\mathcal{HMO}^c(\rr_+^{n+1},\,\cm)};$$
\item[\rm{(ii)}] Let $f\in\mathcal{BMO}^c(\rn,\,\cm)$. Then $u(x,\,t):=f*P_t(x)\in\mathcal{HMO}^c(\rr_+^{n+1},\,\cm)$, and there exists a positive constant $C$ such that
    $$\|u\|_{\mathcal{HMO}^c(\rr_+^{n+1},\,\cm)}
    \leq C\|f\|_{\mathcal{BMO}^c(\rn,\,\cm)}.$$
    \item[\rm{(iii)}] Similarly,  the two assertions above also hold true for the row spaces $\mathcal{HMO}^r(\rr_+^{n+1},\,\cm)$ and $\mathcal{BMO}^r(\rn,\,\cm)$.
\end{enumerate}
\end{theorem}

This result is an operator-valued analogue of the beautiful one due to Fabes {\it et al.} \cite{fjn75}. To prove this result, we need the operator-valued Hardy spaces introduced by Mei \cite{m07}.
Let $P$ be the Poisson kernel of $\rn$:
$$P(y) = c_n
\frac{1}{(|y|^2 + 1)^{(n+1)/2}},$$
{where $c_n$ is the normalizing constant, that is,  $c_n:=\frac{\Gamma(\frac{n+1}{2})}{\pi^{\frac{n+1}{2}}}$.} For $t>0,$ let
$$P_t(y): = \frac{1}{t}
P\left(\frac{{y}}{{t}}\right) = c_n\frac{t}{(|y|^2+t^2)^{(n+1)/2}}.$$
We say that $f$ is an $S_{\cm}$-valued simple function on $\mathbb R^n$ if
 $$f=\sum_{j=1}^Jm_j\otimes\chi_{E_j},$$
where each $m_j\in S_{\cm}$ and $E_j$'s are disjoint measurable  subsets of $\mathbb R^n$ with $|E_j|<\infty$. In what follows, sometimes we call such a $S_{\mathcal M}$-valued simple function a nice function.
For any $S_{\cm}$-valued simple function $f$, its Poisson integral will be denoted by
$$f(y,\,t):=f*P_t(y):=\int_{\rn}
P_t(y-z)f(z)\,dz,\ \  (y,\,t                                             )\in\rr_+^{n+1},$$
and then the Lusin area functions of $f$ are defined by
$$S^c(f)(x):=\lf(\iint_{\Gamma}\lf|\nabla f(y+x,\,t)\r|^2\,dydt\r)^{1/2}\ \
\mathrm{and} \ \
S^r(f)(x):=S^c(f^*)(x),$$
where $$|\nabla f(\cdot,\,\cdot)|^2=\lf(\frac{\partial f}{\partial y_1}\r)^*\frac{\partial f}{\partial y_1}+\cdots+\lf(\frac{\partial f}{\partial y_n}\r)^*\frac{\partial f}{\partial y_n}+\lf(\frac{\partial f}{\partial t}\r)^*\frac{\partial f}{\partial t}$$
and
$\Gamma:=\{(y,\,t)\in\rr^{n+1}_+:|y|<t\}$.
We set
$$\|f\|_{\mathcal{H}_{1}^c(\rn,\,\cm)}:=
\lf\|S^c(f)\r\|_{L_1(L_{\infty}(\rn)\overline{\otimes}\cm)}$$
and
$$\|f\|_{\mathcal{H}_{1}^r(\rn,\,\cm)}:=
\lf\|S^r(f)\r\|_{L_1(L_{\infty}(\rn)\overline{\otimes}\cm)}.$$
The {column operator-valued Hardy space} $\mathcal{H}_{1}^c(\rn,\,\cm)$ (resp. {row operator-valued Hardy space} $\mathcal{H}_{1}^c(\rn,\,\cm)$) is defined to be the completion of the space of all $S_{\cm}$-valued simple functions with finite $\mathcal{H}_{1}^c(\rn,\,\cm)$ (resp. $\mathcal{H}^{r}_1(\rn,\,\cm)$) norm.

The following Fefferman-Stein duality is one of the main results in the above-mentioned paper.

\begin{lemma}\label{l3.1y}
We have
$$(\mathcal{H}_1^{c}(\rn,\,\cm))^*\backsimeq\mathcal{BMO}^c(\rn,\,\cm),$$
in the following sense:
\begin{enumerate}
\item[\rm{(i)}]
Each $g\in \mathcal{BMO}^c(\rn,\,\cm)$ defines a continuous linear
functional $\mathcal{L}_g$ on $\mathcal{H}^c_{1}(\rn,\,\cm)$ by
$$\mathcal{L}_g(f):=\tau\int_{\rn}f(x)g^*(x)\,dx,\ \ \ \mathrm{for}\ \mathrm{any} \  \mathrm{nice} \ f\in \mathcal{H}_{1}^c(\rn,\,\cm);$$
\item[\rm{(ii)}] If $\mathcal{L}\in(\mathcal{H}_{1}^c(\rn,\,\cm))^{*}$, then there exists some
$g\in \mathcal{BMO}^c(\rn,\,\cm)$ such that $\mathcal{L}=\mathcal{L}_g$ as the above.
\end{enumerate}
Moreover, there exists a positive constant $C$ such that
$$C^{-1}\|g\|_{\mathcal{BMO}^c(\rn,\,\cm)}
\leq\|\mathcal{L}_g\|_{(\mathcal{H}_{1}^c(\rn,\,\cm))^{*}}
\leq C\|g\|_{\mathcal{BMO}^c(\rn,\,\cm)}.$$
 Similar result holds for the row spaces.
\end{lemma}

The following lemma is the Kadison-Schwarz inequality, which follows from the operator convexity of $t\rightarrow|t|^2$.
\begin{lemma}\label{l3.0}
 Let $(\Omega,\,d\mu)$  be a measure space.  Assume that $f:\Omega\rightarrow (0,\,\fz)$ and $g:\Omega\rightarrow\cm$ are functions such that all members of the below  inequality make sense. Then we have
\begin{align*}
\lf|\int_{\Omega}fg\,d\mu\r|^2\leq\int_{\Omega}f^2\,d\mu\int_{\Omega}|g|^2\,d\mu,
\end{align*}
where `$\leq$' is understood as the partial order in the positive cone of $\cm$.
\end{lemma}

Now we are ready to show the main result of the present paper.

\begin{proof}[Proof of Theorem \ref{t3.1}] We only need to show Theorem \ref{t3.1} (i), (ii) was proved by Mei in \cite[Lemma 1.4]{m07}.
For the sake of clarity,
now we divide the proof of Theorem \ref{t3.1} (i) into five steps. Let $u\in\mathcal{HMO}^c(\rr_+^{n+1},\,\cm)$.

\textbf{Step 1:} In this step, we show that, for any $(x,\,t)\in\rr_+^{n+1}$ and $i\in[1,\,n]\cap\nn$,
\begin{align*}
\lf|\frac{\partial u}{\partial x_i}(x,\,t)\r|^2\lesssim\frac{1}{t^2}\|u\|^2_{\mathcal{HMO}^c(\rr_+^{n+1},\,\cm)} \ \
\mathrm{and} \ \
\lf|\frac{\partial u}{\partial t}(x,\,t)\r|^2\lesssim\frac{1}{t^2}\|u\|^2_{\mathcal{HMO}^c(\rr_+^{n+1},\,\cm)}.
\end{align*}
Indeed, let $(x_0,\,t_0)\in\rr_+^{n+1}$. By Corollary \ref{c4.9} (i) and Lemma \ref{l3.0}, we have
\begin{align*}
\lf|\frac{\partial u}{\partial x_i}(x_0,\,t_0)\r|^2
&=\lf|\frac{1}{\lf|B((x_0,\,t_0),\,\frac{t_0}{2})\r|}
\iint_{B((x_0,\,t_0),\,\frac{t_0}{2})}\frac{\partial u}{\partial x_i}(x,\,t)\,dxdt\r|^2\\
&\lesssim\frac{1}{t_0^{n+1}}
\iint_{B((x_0,\,t_0),\,\frac{t_0}{2})}\lf|\frac{\partial u}{\partial x_i}(x,\,t)\r|^2\,dydx\\
&\lesssim\frac{1}{t_0^{n+1}}
\int_{B(x_0,\,t_0)}\int_{\frac{t_0}{2}}^{\frac{3t_0}{2}}\lf|\frac{\partial u}{\partial x_i}(x,\,t)\r|^2\,dxdt\\
&\lesssim\frac{1}{t_0^{n+2}}
\int_{B(x_0,\,t_0)}\int_{0}^{2t_0}\lf|\frac{\partial u}{\partial x_i}(x,\,t)\r|^2t\,dxdt\\
&\lesssim\frac{1}{t_0^2}\|u\|_{\mathcal{HMO}^c(\rr_+^{n+1},\,\cm)}^2.
\end{align*}
If we replace $x_i$ by $t$, the above estimate also holds true.
Therefore, we finish the proof of \textbf{Step 1}.

\textbf{Step 2:} We prove that, for any $k\in\nn$, $x\in\rn$ and $t>0$,
$u(\cdot,\,\frac{1}{k})*P_t(x)\in\cm$.
To prove this, we only need to show that, for any $k\in\nn$, $x\in\rn$, $t>0$ and $a\in L_1(\cm)$,
\begin{align}\label{e3.1}
\lf|\lf\langle u(\cdot,\,\frac{1}{k})*P_t(x),\,a\r\rangle\r|\leq C_{(u,k,x,t)}\|a\|_{L_1(\cm)}<\fz.
\end{align}
Firstly, we claim that, for any $(x,\,t)\in\rr_+^{n+1}$,
\begin{align}\label{e3.2}
\lf|\lf\langle u(x,\,t),\,a\r\rangle-\lf\langle u(x_0,\,t),\,a\r\rangle\r|\lesssim
\lf\{
\begin{array}{ll}
 \|a\|_{L_1(\cm)}\|u\|_{\mathcal{HMO}^c(\rr_+^{n+1},\,\cm)}, & \  |x-x_0|\leq t, \\
  \|a\|_{L_1(\cm)}\|u\|_{\mathcal{HMO}^c(\rr_+^{n+1},\,\cm)}
  \log\frac{|x-x_0|}{t}, & \  |x-x_0|> t.
 \end{array}
\r.
\end{align}
In fact, if $|x-x_0|\leq t$, by the mean value theorem, Corollary \ref{c4.9} (ii) and \textbf{Step 1}, then there exists $\xi$ between $x$ and $x_0$ such that
\begin{align*}
\lf|\lf\langle u(x,\,t),\,a\r\rangle-\lf\langle u(x_0,\,t),\,a\r\rangle\r|
=&\lf|\nabla_x\lf\langle u(\xi,\,t),\,a\r\rangle\r||x-x_0|\\
\leq&\|a\|_{L_1(\cm)}\|\nabla_x u(\xi,\,t)\|_{\ell_2^{n}(\cm)}|x-x_0|\\
\lesssim&\|a\|_{L_1(\cm)}\| u\|_{\mathcal{HMO}^c(\rr_+^{n+1},\,\cm)}\frac{|x-x_0|}{t}\\
\lesssim&\|a\|_{L_1(\cm)}\| u\|_{\mathcal{HMO}^c(\rr_+^{n+1},\,\cm)},
\end{align*}
where $\nabla_x:=(\frac{\partial}{\partial x_1},\,\frac{\partial}{\partial x_2},\ldots,\frac{\partial}{\partial x_n})$.

If $|x-x_0|> t$, then we have
\begin{align*}
&\lf|\lf\langle u(x,\,t),\,a\r\rangle-\lf\langle u(x_0,\,t),\,a\r\rangle\r|\\
\leq&\lf|\lf\langle u(x,\,t),\,a\r\rangle-\lf\langle u(x,\,|x-x_0|),\,a\r\rangle\r|+
\lf|\lf\langle u(x,\,|x-x_0|),\,a\r\rangle-\lf\langle u(x_0,\,|x-x_0|),\,a\r\rangle\r|\\
&\ \ \ +\lf|\lf\langle u(x_0,\,|x-x_0|),\,a\r\rangle-\lf\langle u(x_0,\,t),\,a\r\rangle\r|\\
=&:\mathrm{I}_1+\mathrm{I}_2+\mathrm{I}_3.
\end{align*}
For the term $\mathrm{I}_1$, by Corollary \ref{c4.9} (ii) and \textbf{Step 1}, we obtain
\begin{align*}
\mathrm{I}_1=&\lf|\int_{t}^{|x-x_0|}\frac{\partial}{\partial s}
\langle u(x,\,s),\,a\rangle\,ds\r|\\
\leq&\int_{t}^{|x-x_0|}\lf|\frac{\partial}{\partial s}
\langle u(x,\,s),\,a\rangle\r|\,ds\\
\leq&\|a\|_{L_1(\cm)}\int_{t}^{|x-x_0|}\lf\|\frac{\partial u}{\partial s}
(x,\,s)\r\|\,ds\\
\lesssim&\|a\|_{L_1(\cm)}\|u\|_{\mathcal{HMO}^c(\rr_+^{n+1},\,\cm)}
\int_{y}^{|x-x_0|}\frac{1}{s}\,ds\\
\thicksim&\|a\|_{L_1(\cm)}\|u\|_{\mathcal{HMO}^c(\rr_+^{n+1},\,\cm)}
\log\lf(\frac{|x-x_0|}{t}\r).
\end{align*}
For the term $\mathrm{I}_2$, from the mean value theorem, the Fubini theorem, Corollary \ref{c4.9} (ii) and \textbf{Step 1}, we deduce that, there exists $\xi$ between $x$ and $x_0$ such that
\begin{align*}
\mathrm{I}_2=&\lf|\nabla_x\lf\langle u(\xi,\,|x-x_0|),\,a\r\rangle\r||x-x_0|\\
\leq&\lf\|\nabla_x u(\xi,\,|x-x_0|)\r\|\|a\|_{L_1(\cm)}|x-x_0|\\
=&\lf\||\nabla_x u(\xi,\,|x-x_0|)|^2\r\|^{1/2}\|a\|_{L_1(\cm)}|x-x_0|\\
\lesssim&\|a\|_{L_1(\cm)}\|u\|_{\mathcal{HMO}^c(\rr_+^{n+1},\,\cm)}.
\end{align*}
By a similar way in dealing with the estimate $\mathrm{I}_1$, we get that $$\mathrm{I}_3\lesssim\|a\|_{L_1(\cm)}\|u\|_{\mathcal{HMO}^c(\rr_+^{n+1},\,\cm)}
\log\lf(\frac{|x-x_0|}{t}\r),$$
which implies the claim \eqref{e3.2}.

Now we prove \eqref{e3.1}. By the Hille theorem (see Lemma \ref{l2.tt}), we obtain
\begin{align*}
\lf|\lf\langle u(\cdot,\,\frac{1}{k})*P_t(x),\,a\r\rangle\r|
=&\lf|\lf\langle u(\cdot,\,\frac{1}{k}),\,a\r\rangle*P_t(x)\r|\\
\leq&C\int_{\rn}\lf|\lf\langle u(z,\,\frac{1}{k}),\,a\r\rangle\r|
\frac{t}{(t^2+|x-z|^2)^{(n+1)/2}}\,dz\\
=&C\int_{|x-z|\leq\frac{1}{k}}\lf|\lf\langle u(z,\,\frac{1}{k}),\,a\r\rangle\r|
\frac{t}{(t^2+|x-z|^2)^{(n+1)/2}}\,dz
+C\int_{|x-z|>\frac{1}{k}}\cdots\\
=&:\mathrm{J}_1+\mathrm{J}_2.
\end{align*}
For the term $\mathrm{J}_1$, by \eqref{e3.2}, we have
\begin{align*}
\mathrm{J}_1&\leq\int_{|x-z|\leq\frac{1}{k}}\lf|\lf\langle u(z,\,\frac{1}{k}),\,a\r\rangle-\lf\langle u(x,\,\frac{1}{k}),\,a\r\rangle\r|
\frac{t}{(t^2+|x-z|^2)^{(n+1)/2}}\,dz\\
&\ \ \ +\int_{|x-z|\leq\frac{1}{k}}\lf|
\lf\langle u(x,\,\frac{1}{k}),\,a\r\rangle\r|
\frac{t}{(t^2+|x-z|^2)^{(n+1)/2}}\,dz\\
&\leq C\|a\|_{L_1(\cm)}\|u\|_{\mathcal{HMO}^c(\rr_+^{n+1},\,\cm)}
+C_{(x,\,k)}\|a\|_{L_1(\cm)}<\infty.
\end{align*}
For the term $\mathrm{J}_2$, from \eqref{e3.2}, we deduce that
\begin{align*}
\mathrm{J}_2&\leq\int_{|x-z|>\frac{1}{k}}\lf|\lf\langle u(z,\,\frac{1}{k}),\,a\r\rangle-\lf\langle u(x,\,\frac{1}{k}),\,a\r\rangle\r|
\frac{t}{(t^2+|x-z|^2)^{(n+1)/2}}\,dz\\
&\ \ \ +\int_{|x-z|>\frac{1}{k}}\lf|
\lf\langle u(x,\,\frac{1}{k}),\,a\r\rangle\r|
\frac{t}{(t^2+|x-z|^2)^{(n+1)/2}}\,dz\\
&\leq\|a\|_{L_1(\cm)}\|u\|_{\mathcal{HMO}^c(\rr_+^{n+1},\,\cm)}\int_{|x-z|>\frac{1}{k}}
\log(k|x-z|)
\frac{t}{(t^2+|x-z|^2)^{(n+1)/2}}\,dz\\
&\ \ \ +\int_{|x-z|>\frac{1}{k}}\lf|
\lf\langle u(x,\,\frac{1}{k}),\,a\r\rangle\r|
\frac{t}{(t^2+|x-z|^2)^{(n+1)/2}}\,dz\\
&\leq C\|a\|_{L_1(\cm)}\|u\|_{\mathcal{HMO}^c(\rr_+^{n+1},\,\cm)}\int_{|z|>\frac{1}{k}}
\log(k|z|)
\frac{t}{(t^2+|z|^2)^{(n+1)/2}}\,dz\\
&\ \ \ +\int_{|z|>\frac{1}{k}}\lf|
\lf\langle u(x,\,\frac{1}{k}),\,a\r\rangle\r|
\frac{t}{(t^2+|z|^2)^{(n+1)/2}}\,dz\\
&\lesssim C_{(t,\,k)}\|a\|_{L_1(\cm)}\|u\|_{\mathcal{HMO}^c(\rr_+^{n+1},\,\cm)}
+C_{(x,\,k)}\|a\|_{L_1(\cm)}<\infty.
\end{align*}
This finishes the proof of \textbf{Step 2}.

\textbf{Step 3:} In this step, we show that,
for any $k\in\nn$, $x\in\rn$ and $t>0$,
$$u(x,\,t+\frac{1}{k})=u(\cdot,\,\frac{1}{k})*P_t(x).$$
To prove this, we only need to show that, for any $a\in L_1(\cm)$,
$$\lf\langle \nabla u(x,\,t+\frac{1}{k}),\,a\r\rangle=\lf\langle \nabla u(\cdot,\,\frac{1}{k})*P_t(x),\,a\r\rangle.$$
By Corollary \ref{c4.9} (ii) and \textbf{Step 1}, we obtain that, for any $z\in\rn$ and $t>0$,
\begin{align*}
\lf|\nabla\lf\langle u(z,\,t),\,a\r\rangle\r|=&
\lf|\lf\langle \nabla u(z,\,t),\,a\r\rangle\r|\\
\leq&\lf\|\nabla u(z,\,t)\r\|_{\ell_2^{n+1}(\cm)}\|a\|_{L_1(\cm)}\\
\lesssim&\frac{1}{t}\|a\|_{L_1(\cm)}\|u\|_{\mathcal{HMO}^c(\rr_+^{n+1},\,\cm)}
\end{align*}
From this fact and applying the maximum principle for the scalar-valued harmonic function $\nabla\lf\langle  u(x,\,t+\frac{1}{k}),\,a\r\rangle$,
 we deduce that, for any $x\in\rn$, $t>0$ and $k\in\zz_+$, $\nabla\lf\langle  u(x,\,t+\frac{1}{k}),\,a\r\rangle
=\nabla\lf\langle  u(\cdot,\,\frac{1}{k}),\,a\r\rangle*P_t(x)$.
Therefore, it follows from Corollary \ref{c4.9} (ii) and the Hille theorem (see Lemma \ref{l2.tt}) that
\begin{align*}
\lf\langle \nabla u(x,\,t+\frac{1}{k}),\,a\r\rangle
&=\nabla\lf\langle  u(x,\,t+\frac{1}{k}),\,a\r\rangle\\
&=\nabla\lf\langle  u(\cdot,\,\frac{1}{k}),\,a\r\rangle*P_t(x)\\
&=\lf\langle \nabla u(\cdot,\,\frac{1}{k})*P_t(x),\,a\r\rangle.
\end{align*}
This completes the proof of \textbf{Step 3}.

\textbf{Step 4:} In this step, we prove that
\begin{align*}
\lf\|u(\cdot,\,\cdot+\frac{1}{k})\r\|_{\mathcal{HMO}^c(\rr_+^{n+1},\,\cm)}
\lesssim \|u\|_{\mathcal{HMO}^c(\rr_+^{n+1},\,\cm)}, \ \ \mathrm{for \ any} \ k\in\nn,
\end{align*}
where the constant is independent of $k$.
In fact, for any $x_0\in\rn$ and $r_0>0$. If $r_0>\frac{1}{k}$, then
 \begin{align*}
&\lf\|\frac{1}{|B(x_0,\,r_0)|}\int_{B(x_0,\,r_0)}\int_0^{r_0}\lf|\nabla
u(x,\,t+\frac{1}{k})\r|^2t\,dtdx\r\|\\
\lesssim&\lf\|\frac{1}{|B(x_0,\,2r_0)|}\int_{B(x_0,\,2r_0)}\int_0^{2r_0}\lf|\nabla
u(x,\,s)\r|^2s\,dsdx\r\|\\
\lesssim&\|u\|_{\mathcal{HMO}^c(\rr_+^{n+1},\,\cm)}^2.
\end{align*}
If $r_0\leq\frac{1}{k}$, by \textbf{Step 1}, we have
\begin{align*}
&\lf\|\frac{1}{|B(x_0,\,r_0)|}\int_{B(x_0,\,r_0)}\int_0^{r_0}\lf|\nabla
u(x,\,t+\frac{1}{k})\r|^2t\,dtdx\r\|\\
\lesssim&\|u\|_{\mathcal{HMO}^c(\rr_+^{n+1},\,\cm)}^2
\int_0^{r_0}\frac{t}{(t+\frac{1}{k})^2}
\,dt\\
\lesssim&\|u\|_{\mathcal{HMO}^c(\rr_+^{n+1},\,\cm)}^2.
\end{align*}

\textbf{Step 5:} In this step, we obtain that, for all $k\in\nn$,
\begin{align*}
\lf\|u(\cdot,\,\frac{1}{k})\r\|_{\mathcal{BMO}^c(\rn,\,\cm)}
\lesssim \|u\|_{\mathcal{HMO}^c(\rr_+^{n+1},\,\cm)}.
\end{align*}
To prove this, now we claim that, for any $S_{\cm}$-valued simple function $f\in\mathcal{H}_1^c(\rn,\,\cm)$,
\begin{align*}
\lf|\tau\int_{\rn}f(x)u^*(x,\,\frac{1}{k})\r|
\lesssim \|f\|_{\mathcal{H}_1^c(\rn,\,\cm)}\|u\|_{\mathcal{HMO}^c(\rr_+^{n+1},\,\cm)},
\ \mathrm{ \ uniformly\ for}\ k\in\nn.
\end{align*}
Combining \textbf{Step 4}, we know that this claim can be proved in a similar way as \cite[Theorem 2.4 (i)]{m07}. To control the length of the paper, we omit its details.
Therefore, by Lemma \ref{l3.1y}, we obtain that  $u(\cdot,\,\frac{1}{k})\in\mathcal{BMO}^c(\rn,\,\cm)$ and
$$\lf\|u(\cdot,\,\frac{1}{k})\r\|_{\mathcal{BMO}^c(\rn,\,\cm)}
=\sup_{\|f\|_{\mathcal{H}_1^c(\rn,\,\cm)}\leq1}\lf|\tau\int_{\rn}f(x)u^*(x,\,\frac{1}{k})\r|
\leq C\|u\|_{\mathcal{HMO}^c(\rr_+^{n+1},\,\cm)},$$
which implies that $\{u(\cdot,\,\frac{1}{k})\}_{k\in\nn}$ is an uniformly bounded sequence in $\mathcal{BMO}^c(\rn,\,\cm)$. From this and the Banach-Alaoglu theorem, we obtain that there exists a subsequence $\{u(\cdot,\,\frac{1}{k_j})\}_{j\in\nn}$ and some $g\in\mathcal{BMO}^c(\rn,\,\cm)$ such that
\begin{align*}
u(\cdot,\,\frac{1}{k_j})\rightarrow g,\ \ \mathrm{in\ the\ weak-^*\ sense}, \ \mathrm{as}\ j\rightarrow\infty
\end{align*}
and
\begin{align*}
\|g\|_{\mathcal{BMO}^c(\rn,\,\cm)}
\leq C\|u\|_{\mathcal{HMO}^c(\rr_+^{n+1},\,\cm)}.
\end{align*}
From \textbf{Step 3}, we know that, for any $x\in\rn$, $t>0$, $j\in\nn$ and $a\in L_1(\cm)$,
\begin{align}\label{e2.00c}
\nabla\lf\langle  u(x,\,t+\frac{1}{k_j}),\,a\r\rangle=\nabla \lf\langle u(\cdot,\,\frac{1}{k_j})*P_t(x),\,a\r\rangle.
\end{align}
Next we calculate the limit of both sides of the above equality.
Denote the left-hand side of \eqref{e2.00c} by $\mathrm{LHS}$, the right-hand side of \eqref{e2.00c} by $\mathrm{RHS}$.
By Corollary \ref{c4.9} (ii), it follows that, for any given $x\in\rn$ and $t>0$,
\begin{align*}
\lim_{j\rightarrow+\infty} \mathrm{LHS}=\lim_{j\rightarrow+\infty}\nabla\lf\langle  u(x,\,t+\frac{1}{k_j}),\,a\r\rangle=\nabla\lf\langle  u(x,\,t),\,a\r\rangle=\lf\langle\nabla u(x,\,t),\,a\r\rangle.
\end{align*}
Note that for any $x\in\rn,\,t>0$, each coordinate of $\nabla P_t(x-\cdot)a^*$ is an element of $\mathcal{H}^c_1(\rn,\,\cm)$. Then since $u(\cdot,\,\frac{1}{k_j})\rightarrow g$ converges to $\mathcal{BMO}^c(\rn,\,\cm)$ in the weak-$^*$ topology as $j\rightarrow+\infty$, we deduce that, for any given $x\in\rn$ and $t>0$,
\begin{align*}
\lim_{j\rightarrow+\infty}\mathrm{RHS}&=\lim_{j\rightarrow+\infty}\nabla \lf\langle u(\cdot,\,\frac{1}{k_j})*P_t(x),\,a\r\rangle\\
&=\lim_{j\rightarrow+\infty}\tau\lf(u(\cdot,\,\frac{1}{k_j})*\nabla P_t(x)a\r)\\
&=\lim_{j\rightarrow+\infty}\tau\int_{\rn}u(z,\,\frac{1}{k_j})\nabla P_t(x-z)a\,dz\\
&=\lim_{j\rightarrow+\infty}\overline{\tau\int_{\rn}\nabla P_t(x-z)a^*u^*(z,\,\frac{1}{k_j})\,dz}\\
&=\overline{\tau\int_{\rn}\nabla P_t(x-z)a^*g^*(z)\,dz}\\
&=\tau\int_{\rn}g(z)\nabla P_t(x-z)a\,dz
=\lf\langle\nabla g*P_t(x),\,a\r\rangle.
\end{align*}
Therefore, from the arbitrariness of $a\in L_1(\cm)$, we further conclude that, for any $(x,\,t)\in\rr_+^{n+1}$,
$$\nabla u(x,\,t)=\nabla g*P_t(x),$$ and so $u= g*P_t+C$, that is
$u(x,\,t)=g*P_t(x)$ as an element of $\mathcal{HMO}^c(\rr_+^{n+1},\,\cm)$.
This completes the proof of Theorem \ref{t3.1}.
\end{proof}


\bigskip
\noindent

Cheng Chen

\medskip

\noindent
 Department of Mathematics\\
 Sun Yat-sen(Zhongshan) University\\
 Guangzhou, 510275, P.R. China\\
\noindent

\medskip
\noindent

\noindent{E-mail}:
\texttt{chench575@mail.sysu.edu.cn}

\medskip

Guixiang Hong

\medskip

\noindent
Institute for Advanced Study in Mathematics\\
 Harbin Institute of Technology\\
Harbin, 150001, P.R. China
\noindent

\medskip
\noindent
\noindent{E-mail }:
\texttt{gxhong@hit.edu.cn} \\

\medskip

Wenhua Wang

\medskip
\noindent
Institute for Advanced Study in Mathematics\\
 Harbin Institute of Technology \\
Harbin, 150001, P.R. China
\smallskip

\noindent{E-mail }:
\texttt{whwangmath@whu.edu.cn} \\
 \medskip

\noindent

\noindent

\medskip
\noindent


\smallskip

\bigskip \medskip

\end{document}